\documentclass[11pt,bezier]{article}
\usepackage{amsmath,amssymb,amsfonts,euscript,graphicx,fncychap,setspace,inputenc,fontenc,sidecap,float,epstopdf,tikz,animate,color}

\usepackage[margin=.75in]{geometry}

\title{{\bf On the Diameter and Girth of an Annihilating-Ideal Graph\thanks {This research was in part supported 
 by grant numbers (89160031) and (91130031) from IPM for the second and fourth authors respectively.}}}
\author{{{\bf F. Aliniaeifard${^{\rm a}}$,   M. Behboodi${^{\rm b}}$\thanks
{ Corresponding author}}},  {\bf E. Mehdi-Nezhad}${^{\rm c}}$, {\bf Amir M. Rahimi}${^{\rm d}}$\\
{\small{$^{\rm a}$Department of Mathematics and Statistics, York University }} \\
{\small{North York, Ontario Canada M3J 1P3}}\\
  {\small{${^{\rm b}}$Department of Mathematical Sciences, Isfahan University of Technology}}\vspace{-1mm}\\ {\small{
Isfahan,
Iran, 84156-83111}}\\
{\small{$^{\rm c}$Department of Mathematics, University of Cape Town }} \\
{\small{Cape Town, South Africa}}\\
{\small{$^{\rm b,d}$School of Mathematics, Institute for Research in Fundamental Sciences (IPM) }} \\
{\small{Tehran Iran, 19395-5746}}\\
\footnotesize{${^{\rm a}}$faridanf@mathstat.yorku.ca}\vspace{-1mm}\\
\footnotesize{${^{\rm b}}$mbehbood@cc.iut.ac.ir}\vspace{-1mm}\\
\footnotesize{${^{\rm c}}$mhdelh001@myuct.ac.ca}\vspace{-1mm}\\
\footnotesize{${^{\rm d}}$amrahimi@ipm.ir}}

\def\be{\begin{enumerate}}
\def\ee{\end{enumerate}}

\newtheorem{ttheo}{Theorem}[section]
\newtheorem{ccoro}[ttheo]{Corollary}
\newtheorem{llem}[ttheo]{Lemma}
\newtheorem{eexam}[ttheo]{Example}
\newtheorem{rrem}[ttheo]{Remark}
\newtheorem{ppro}[ttheo]{Proposition}

\newenvironment{pproof}{\noindent{\bf Proof. }}{}

\date{}

\begin{document}
  \maketitle
\begin{abstract}
{Let $R$ be a commutative ring with $1\neq 0$ and $\Bbb{A}(R)$ be the set of ideals with nonzero annihilators.
The annihilating-ideal graph of $R$ is defined as the graph $\Bbb{AG}(R)$ with the vertex set $\Bbb{A}(R)^{*} = \Bbb{A}(R)\setminus \{(0)\}$ and two distinct vertices $I$ and $J$ are adjacent if and only if $IJ = (0)$.
In this paper, we first study the interplay
between the diameter of annihilating-ideal graphs
and zero-divisor graphs. Also, we  characterize rings $R$ when ${\rm gr}(\Bbb{AG}(R))\geq 4$, and so we characterize rings whose annihilating-ideal graphs are bipartite. Finally, in the last section we discuss on a relation between the Smarandache vertices and diameter of $\Bbb {AG}(R)$.\vspace{3mm}
  }
\end{abstract}

 \section{ Introduction}

Throughout this paper, all rings
are assumed to be commutative with identity. We denote the set
of all ideals which are a subset of an ideal $J$ of $R$ by $\Bbb{I}(J)$. We call an
ideal $I$ of $R$, an {\it annihilating-ideal} if there exists a
non-zero ideal $J$ of $R$ such that $IJ=(0)$, and use the
notation $\Bbb{A}(R)$ for the set of all annihilating-ideals of
$R$. By the {\it Annihilating-Ideal
graph} $\Bbb{AG}(R)$ of $R$ we mean the graph with vertices
$\Bbb{AG}(R)^{*} = \Bbb{A}(R) \setminus \{(0)\}$ such that
there is an (undirected) edge between vertices $I$ and $J$ if and
only if $I\neq J$ and $IJ=(0)$. Thus $\Bbb{AG}(R)$ is an empty
graph if and only if $R$ is an integral domain. The concept of the annihilating-ideal graph of a commutative ring was first introduced by Behboodi and Rakeei in \cite{be-ra1} and \cite{be-ra2}. Also in \cite{abmr}, the authors of this paper have extended and studied this notion to a more general setting as {\it the annihilating-ideal graph with respect to an ideal of $R$}, denoted $\Bbb{AG}_I(R)$. 

Let $G$ be a graph.
Recall that $G$ is connected if there is a path between any two
distinct vertices of $G$. For vertices
$x$ and $y$ of $G$, let $d(x, y)$ be the length of a shortest path
from $x$ to $y$ ($d(x, x) = 0$ and $d(x, y) = \infty$ if there is
no such path). The diameter of $G$, denoted by ${\rm diam}(G)$, is $sup\{d(x, y) | x$
and $y$ are vertices of $G\}$. The girth of $G$, denoted by
${\rm gr}(G)$, is the length of a shortest cycle in $G$ (${\rm gr}(G) =
\infty$ if $G$ contains no cycles). $\Bbb{AG}(R)$ is connected
with ${\rm diam}(\Bbb{AG}(R))\leq 3$ {\rm \cite[Theorem 2.1]{be-ra1}} and if $\Bbb{AG}(R)$ contains a cycle, then
${\rm gr}(\Bbb{AG}(R))\leq 4$ {\rm \cite[Theorem 2.1]{be-ra1}}. Thus ${\rm diam}(\Bbb{AG}(R))=0,1,2$ or 3; and
${\rm gr}(\Bbb{AG}(R))=3,4$ or $\infty$. Also,
$\Bbb{AG}(R)$ is a singleton (i.e., ${\rm diam}(\Bbb{AG}(R)) = 0$) if and only if either $R\cong \frac{\Bbb{K}[x]}{(x^2)}$, where $\Bbb{K}$ is a field or $R\cong L$, where $L$ is a coefficient ring of characteristic $p^2$, that is $L\cong \frac{A}{(p^2.1)}$, where $A$ is a discrete valuation ring of characteristic $0$ and residue field of characteristic $p$, for some prime number $p$ {\rm \cite[Remark 10]{aal}}.

Let $Z(R)$ be the set of
zero-divisors of $R$. The zero-divisor graph of $R$, denoted by
$\Gamma(R)$, is the (undirected) graph with vertices $Z(R)^{*} =
Z(R) \setminus \{0\}$, the set of nonzero zero-divisors of $R$, and
for distinct $x$, $y$ $\in Z(R)^{*}$, the vertices $x$ and $y$ are
adjacent if and only if $x y = 0$. Note that $\Gamma(R)$ is the
empty graph if and only if $R$ is an integral domain. Moreover, a
nonempty $\Gamma(R)$ is finite if and only if $R$ is finite and
not a field {\rm \cite[Theorem 2.2]{an-li}}. The concept of a zero-divisor graph
was introduced by Beck \cite{beck}.
However, he let all the elements of $R$ be vertices of the graph
and was mainly interested in colorings. $\Gamma(R)$ is connected with
${\rm diam}(\Gamma(R))\leq 3$ {\rm \cite[Theorem 2.3]{an-li}} and if $\Gamma(R)$ contains a cycle, then
${\rm gr}(\Gamma(R))\leq 4$ {\rm \cite[Theorem 2.2(c)]{an-li2}}. Thus ${\rm diam}(\Gamma(R))=0,1,2$ or
3; and ${\rm gr}(\Gamma(R))=3,4$ or $\infty $.
For a ring $R$, $nil(R)$ is the set of the nilpotent elements of
$R$. We say that $R$ is reduced if $nil(R)=0$.

Let $K_n$ denote the complete graph on
$n$ vertices. That is, $K_n$ has vertex set $V$ with $|V| = n$ and
$a-b$ is an edge for every $a,b \in V$. Let
$K_{m,n}$ denote the complete bipartite graph. That is, $K_{m,n}$
has vertex set $V$ consisting of the disjoint union of two subsets,
$V_{1}$ and $V_2$, such that $|V_{1}| = m$ and $|V_{2}| = n$, and
$a-b$ is an edge if and only if $a \in V_{1}$ and
$b\in V_{2}$. We may sometimes write $K_{|V_{1}|,|V_{2}|}$ to
denote the complete bipartite graph with vertex sets $V_{1}$ and
$V_{2}$. Note that $K_{m,n} = K_{n,m}$. Also, for every positive integer $n$, we denote a path of order $n$, by $P_n$.

In the present paper, we study the diameter and girth of annihilating-ideal graphs. In Section
$2$, we show that if $R$ is a Noetherian ring with $\Bbb{AG}(R)\ncong K_{2}$,
then ${\rm diam}(\Bbb{AG}(R)) = {\rm diam}(\Bbb{AG}(R[x]) =
{\rm diam}(\Bbb{AG}(R[x_{1},x_{2},..., x_{n}])= {\rm diam}(\Bbb{AG}(R[[x]]) = {\rm diam}(\Gamma(R)) = {\rm diam}(\Gamma(R[x]) =
{\rm diam}(\Gamma(R[x_{1},x_{2},..., x_{n}])= \Gamma(R[[x]])$. In Section $3$, we characterize rings $R$ when ${\rm gr}(\Bbb{AG}(R))\geq 4$. Finally, in the last section, we study some properties of the {\it Smarandache vertices} of $\Bbb {AG}(R)$.
\vspace{3mm}

\section{ Diameter of $\Bbb{AG}(R)$, $\Bbb{AG}(R[x])$, and $\Bbb{AG}(R[[x]])$}

In this section, we show that if $R$ is a Noetherian ring with $\Bbb{AG}(R)\ncong K_{2}$,
then ${\rm diam}(\Bbb{AG}(R)) = {\rm diam}(\Bbb{AG}(R[x])) = {\rm diam}(\Bbb{AG}(R[x_{1},x_{2},..., x_{n}]))= {\rm diam}(\Bbb{AG}(R[[x]])) = {\rm diam}(\Gamma(R)) = {\rm diam}(\Gamma(R[x])) = {\rm diam}(\Gamma(R[[x]]))$.\\

We remark that if $R$ is a commutative ring with identity, then the set
of regular elements of $R$ forms a saturated and multiplicatively
closed subset of $R$. Hence the collection of zero-divisors of
$R$ is the set-theoretic union of prime ideals. We write
$Z(R)=\cup_{i\in \Lambda}P_{i}$ with each $P_{i}$ prime. We will
also assume that these primes are maximal with respect to being
contained in $Z(R)$.

\begin{ttheo}\label{1.15}
Let $R$ be a ring  and $\Bbb{AG}(R)\not \cong K_{2}$. Then the following
conditions are equivalent:

${\rm (1)}$ $\Bbb{AG}(R)$ is a complete graph.

${\rm (2)}$ $\Bbb{AG}(R[x])$ is a complete graph.

${\rm (3)}$ $\Bbb{AG}(R[x_{1},x_{2},...,x_{n}])$ for all $n > 0$ is a complete graph.

${\rm (4)}$ $\Bbb{AG}(R[[x]])$ is a complete graph.

${\rm (5)}$ $\Gamma(R)$ is a complete graph.

${\rm (6)}$ $\Gamma(R[x])$ is a complete graph.

${\rm (7)}$ $\Gamma(R[x_{1},x_{2},...,x_{n}])$ for all $n > 0$ is a complete graph.

${\rm (8)}$ $\Gamma(R[[x]])$ is a complete graph.

${\rm (9)}$ $(Z(R))^{2}=0$.
\end{ttheo}

\begin{pproof}
  If $R\cong \Bbb{Z}_2\times \Bbb{Z}_2$, then  $\Bbb{AG}(R)\cong K_2$, yielding a contradiction. Thus By {\rm \cite[Theorem 2.8]{an-li}}, $\Gamma(R)$ is a complete graph if and only if   $(Z(R))^{2}=(0)$. Also, by {\rm \cite[Theorem 3]{aal}}, $\Bbb{AG}(R)$ is a complete graph if and only if $(Z(R))^{2}=(0)$. So, the results follow easily from {\rm \cite[Theorem 3.2]{axt}}.\hfill $\square$
\end{pproof}

\begin{llem}\label{1.10}
Let $R$ be a ring such that
${\rm diam}(\Bbb{AG}(R)) = 2$. If $Z(R) = P_{1} \cup P_{2}$ with $P_{1}$
and $P_{2}$ distinct primes in $Z(R)$, then $P_{1}\cap
P_{2} = (0)$.
\end{llem}

\begin{pproof}
Let $x\in P_{1} \cap P_{2}$, $p_1\in P_1\setminus P_2$ and $p_2\in P_2\setminus P_1$. Since ${\rm diam}(\Bbb{AG}(R))=2$, either $(Rp_1)(Rp_2)=(0)$ or there exists a non-zero ideal $I$ such that $I\subseteq {\rm Ann}(p_1)\cap {\rm Ann}(p_2)$. If $(Rp_1)(Rp_2)\neq (0)$, then $I(p_1+p_2)=(0)$. Thus $p_1+p_2\in Z(R)$, yielding a contradiction. Therefore, $(Rp_1)(Rp_2)=(0)$ and so $p_1p_2=0$. Since $p_2+x\in P_2\setminus P_1$ and $p_1+x\in P_1\setminus P_2$, we conclude that $0=p_1(p_2+x)=p_1x$ and $0=p_2(p_1+x)=p_2x$. Thus $x(p_1+p_2)=0$, so $x=0$. Therefore, $P_1\cap P_2=(0)$.\hfill $\square$
\end{pproof}

\begin{llem}\label{new}
Let $R$ be a Noetherian ring and $\Bbb{AG}(R)\not \cong K_{2}$. Then ${\rm diam}(\Bbb{AG}(R))=2$ if and only if $Z(R)$ is either the union of two primes with
intersection $(0)$ or $Z(R)$ is a prime ideal such that
$(Z(R))^{2} \neq (0)$.
\end{llem}

\begin{pproof}
Suppose that $Z(R)=\cup_{i\in
\Lambda}P_{i}$ where every $P_i$ is a maximal prime in $Z(R)$ and $|\Lambda|>2$. Since $R$ is a Noetherian ring, by {\rm \cite[Theorem 80]{kap}}, $\Lambda$ is finite. Let $P_1, P_2, P_3 \in \{P_i: i\in \Lambda\}$. If $P_1 \subseteq \cup_{i\in \Lambda\setminus \{1\}}P_{i}$, then by {\rm \cite[Theorem 81]{kap}}, $P_1\subseteq P_i$ for some $i\in \Lambda\setminus \{1\}$, yielding a contradiction. Therefore, there exists $p_1\in P_1\setminus \cup_{i\in \Lambda\setminus \{1\}}P_{i}$. Similarly, there exists $p_2\in P_2\setminus \cup_{i\in \Lambda \setminus \{2\}}P_{i}$. Since ${\rm diam}(\Bbb{AG}(R))=2$, either $(Rp_1)(Rp_2)=(0)$ or there exists a non-zero ideal $I$ such that $I\subseteq {\rm Ann}(p_1)\cap {\rm Ann}(p_2)$. If $(Rp_1)(Rp_2)\neq (0)$, then $I(Rp_1+Rp_2)=(0)$. Thus there exists $P_k \in \{P_i: i\in \Lambda\}$ such that $Rp_1+Rp_2\subseteq P_k$, yielding a contradiction. Therefore, $(Rp_1)(Rp_2)=(0)$ and so $p_1p_2=0$. Thus $p_1p_2\in P_3$, so $p_1\in P_3$ or $p_2\in P_3$, yielding a contradiction.
Thus $|\Lambda|\leq 2$. If $(Z(R))^2=(0)$, then since $\Bbb{AG}(R)\not \cong K_2$, by Theorem \ref{1.15}, ${\rm diam}(\Bbb{AG}(R))\leq 1$, yielding a contradiction. We conclude that $Z(R)=P_1\cup P_2$. Then by Lemma \ref{1.10}, $P_1\cap P_2=(0)$. Thus $Z(R)$ is either the union of two primes with intersection $(0)$ or $Z(R)$ is prime such that $(Z(R))^{2} \neq (0)$.

Conversely, if $Z(R)=P$ is a prime ideal, then by {\rm \cite[Theorem 82]{kap}}, there exists a nonzero element $a\in R$ such that $aZ(R)=(0)$. Let $I, J\in V(\Bbb{AG}(R))$. Then $(Ra)I=(Ra)J=(0)$. Therefore, ${\rm diam}(\Bbb{AG}(R)) \leq 2$. If ${\rm diam}(\Bbb{AG}(R) )\leq 1$, then since $\Bbb{AG}(R)\not \cong K_2$, by Theorem \ref{1.15}, $(Z(R))^2=(0)$, yielding a contradiction. Thus ${\rm diam}(\Bbb{AG}(R)) = 2$. Now, we assume that $Z(R)$ is the union of two primes with
intersection $(0)$. Let $Z(R)=P_1\cup P_2$ and $I, J\in V(\Bbb{AG}(R))$. Since $I\subseteq Z(R)=P_1\cup P_2$, by {\rm \cite[Theorem 81]{kap}}, $I\subseteq P_1$ or $I\subseteq P_2$. Similarly, $J\subseteq P_1$ or $J\subseteq P_2$. Without loss of generality we can assume that $J\subseteq P_1$. If $I\subseteq P_2$, then $IJ\subseteq P_1P_2=(0)$. If $I\subseteq P_1$, then $IP_2=JP_2=(0)$. Therefore, ${\rm diam}(\Bbb{AG}(R) )\leq 2$. If ${\rm diam}(\Bbb{AG}(R) )\leq 1$, then since $\Bbb{AG}(R)\not \cong K_2$, by Theorem \ref{1.15}, $Z(R)$ is a prime ideal such that $(Z(R))^2=(0)$, yielding a contradiction. Thus ${\rm diam}(\Bbb{AG}(R)) = 2$.
\end{pproof}

\begin{ttheo}
Let $R$ be a Noetherian ring and $\Bbb{AG}(R)\not \cong K_{2}$. Then the following
conditions are equivalent:

${\rm (1)}$ ${\rm diam}(\Bbb{AG}(R)) = 2$.

${\rm (2)}$ ${\rm diam}(\Bbb{AG}(R[x])) = 2$.

${\rm (3)}$ ${\rm diam}(\Bbb{AG}(R[x_{1},x_{2},...,x_{n}])) = 2$ for all $n > 0$.

${\rm (4)}$ ${\rm diam}(\Bbb{AG}(R[[x]])) = 2$.

${\rm (5)}$ ${\rm diam}(\Gamma(R)) = 2$.

${\rm (6)}$ ${\rm diam}(\Gamma(R[x])) = 2$.

${\rm (7)}$ ${\rm diam}(\Gamma(R[x_{1},x_{2},...,x_{n}])) = 2$ for all $n > 0$.

${\rm (8)}$ ${\rm diam}(\Gamma(R[[x]])) = 2$.

${\rm (9)}$ $Z(R)$ is either the union of two primes with intersection $(0)$, or $Z(R)$ is prime and
$(Z(R))^{2}\neq 0$.
\end{ttheo}

\begin{pproof}
It follows from Lemma \ref{new} and {\rm \cite[Theorem 3.11]{axt}}.\hfill $\square$
\end{pproof}

\section {A Characterization of the Ring $R$ When ${\rm gr}(\Bbb{AG}(R))\geq
4$}
In \cite[Section 3]{aal1}, the authors have studied rings whose annihilating-ideal graphs are bipartite. In this section, we characterize rings $R$ when ${\rm gr}(\Bbb{AG}(R))\geq 4$, and so we characterize rings whose annihilating-ideal graphs are bipartite.

\begin{ppro}\label{cnew}
Let $R$ be a reduced ring. Then the following statements are equivalent.

{\rm (1)} $Z(R)$ is the union of two primes with intersection $(0)$.

{\rm (2)} $\Bbb{AG}(R)$ is a complete bipartite graph.
\end{ppro}

\begin{pproof}
${\rm (1)} \Rightarrow {\rm (2)}$ Let $Z(R)=P_1\cup P_2$, where $P_1$ and $P_2$ are prime ideals such that $P_1\cap P_2=(0)$. Therefore, $P_1P_2=(0)$. Since ${\rm Ann}(P_1)P_1=(0)\subseteq P_2$ and $P_1 \nsubseteq P_2$, we conclude that ${\rm Ann}(P_1)\subseteq P_2$. Note that $P_2\subseteq {\rm Ann}(P_1)$, so ${\rm Ann}(P_1)=P_2$. Similarly ${\rm Ann}(P_2)=P_1$. Let $V_1=\Bbb{I}(P_1)$ and $V_2=\Bbb{I}(P_2)$. Let $J_1,J_2\in V_2$ be two nonzero ideals such that $J_1J_2=(0)$. Then $J_1J_2\subseteq P_1$, and so $J_1\in P_1$ or $J_2\in P_1$, yielding a contradiction since $P_1\cap P_2=(0)$. Therefore, every non-zero ideals $J_1$ and $J_2$ in $V_2$ are not adjacent. Similarly, every  $I_1$ and $I_2$ in $ V_1$  are not adjacent. Since $P_1\cap P_2=(0)$, ${\rm Ann(P_1)}=P_2$ and ${\rm Ann(P_2)}=P_1$, we can conclude that $\Bbb{AG}(R)\cong K_{|V_1|-1,|V_2|-1}$.

${\rm (2)} \Rightarrow {\rm (1)}$ Let $\Bbb{AG}(R)\cong K_{|V_1|,|V_{2}|}$ such that $V_1=\{I_{i}\in \Bbb{A}(R): i\in \Lambda_{1}\}\setminus \{(0)\}$ and $V_2=\{J_{j}\in \Bbb{A}(R): j\in \Lambda_{2}\}\setminus \{(0)\}$. Let $P_{1}=\bigcup_{i\in \Lambda_{1}}I_{i} $ and
$P_{2}=\bigcup_{j\in \Lambda_{2}}J_{j} $. Therefore, $Z(R)=P_{1}\cup
P_{2}$. Let $a_1,a_2\in P_1$. Then there exist ideals $I_1,I_2\subseteq P_1$ such that $a_1\in I_1$ and $a_2\in I_2$. Since for every ideal $J_j\subseteq P_2$, $J_j(R(a+b))\subseteq J(I_1+I_2)=(0)$, we conclude that $R(a+b)\subseteq P_1$. Thus $a+b\in P_1$. Also, it is easy to see that for every $r\in R$ and $a\in P_1$, $ra\in P_1$, so $P_1$ is an ideal. Similarly $P_2$ is an ideal. Let $P_{1}\cap P_{2}\neq (0)$. Since $P_1P_2=(0)$, $(P_{1}\cap P_{2})Z(R)= (P_{1}\cap
P_{2})(P_{1}\cup P_{2})=(0)$. Thus $Z(R)$ is an ideal, yielding a contradiction. Therefore,
$P_{1}\cap P_{2}=0$. Now, we show that $P_1$ and $P_2$ are prime ideals. Let $ab\in P_1$ and $a,b\not \in P_1$. Since $ab\in Z(R)$, $a\in Z(R)$ or $b\in Z(R)$. Without loss of generality we assume that $a\in Z(R)$. Since $Z(R)=P_1\cup P_2$ and $a\not \in P_1$, we conclude that $a\in P_2$. Hence $ab\in P_2$. Since $ab\in P_1\cap P_2=(0)$, $ab=0$. If $Ra=Rb$, then $a^2=0$, yielding a contradiction since $R$ is a reduced ring. Thus $Ra\neq Rb$. Since $Ra\subseteq P_2$, $Ra\in \{J_{j}: j\in
\Lambda_{2}\}$. Hence $Rb\in \{I_{i}: i\in \Lambda_{1}\}$. Thus $Rb\in P_1$, yielding a contradiction since $b\not \in P_1$. Therefore, $P_1$ is a prime ideal. Similarly $P_2$ is a prime ideal. So, $Z(R)=P_1\cup P_2$, where $P_1$ and $P_2$ are prime ideals such that $P_1\cap P_2=(0)$. \hfill$\square$
\end{pproof}

\begin{ttheo}\label{3.4}
The following statements are equivalent for a reduced ring $R$.

{\rm (1)} ${\rm gr}(\Bbb{AG}(R))=4$.

{\rm (2)} $\Bbb{AG}(R)\cong K_{|V_1|,|V_2|}$, where $|V_1|,|V_2|\geq 2$.

 {\rm (3)} $Z(R)$ is the union of two primes $P_1$ and $P_2$ with intersection $(0)$ and $|\Bbb{I}(P_1)|,|\Bbb{I}(P_2)|\geq 3$.
\end{ttheo}

\begin{pproof}
${\rm (1)}\Rightarrow {\rm (2)}$ First, we show that ${\rm diam}(\Bbb{AG}(R))=2$. If ${\rm diam}(\Bbb{AG}(R))=0$ or $1$, then $\Bbb{AG}(R)$ is a complete graph and so ${\rm gr}(\Bbb{AG}(R))$ is $3$ or $\infty$, yielding a contradiction. If ${\rm diam}(\Bbb{AG}(R))=3$, then there exist $I_1, I_2, I_3, I_4\in \Bbb{A}(R)$ such that $I_1-I_2-I_3-I_4$, $I_1I_3\neq (0)$, $I_2I_4\neq (0)$ and $I_1I_4\neq (0)$. If $I_1I_4= I_2$, then since $(I_1I_4)I_2=(0)$, $(I_2)^{2}=(0)$, yielding a contradiction. Similarly $I_1I_4\neq I_3$. Thus $I_2-I_3-I_1I_4-I_2$ is a cycle and so ${\rm gr}(\Bbb{AG}(R))=3$, yielding a contradiction. Therefore, ${\rm diam}(\Bbb{AG}(R))=2$. We now show that $\Bbb{AG}(R)$ is a complete bipartite graph. Since ${\rm gr}(\Bbb{AG}(R))=4$, there exist $I,J,K,L\in \Bbb{A}(R)$ such that $I-J-K-L-I$. We show that $\Bbb{AG}(R)\cong K_{|V_1|, |V_2|}$, where $V_1=\{T\in \Bbb{A}(R)^{*}: T\subseteq {\rm Ann}(I)\}$ and $V_2=\{S\in \Bbb{A}(R)^{*}: S\nsubseteq {\rm Ann}(I)\}$. Let $T, T_1\in V_1$ and $S, S_1\in V_2$. Then $IT=(0)$ and $IS\neq (0)$. Assume that $TS\neq (0)$. Since ${\rm diam}(\Bbb{AG}(R))=2$, there exists $H\in \Bbb{A}(R)$ such that $I-H-S$. If $TS=H$ or $TS=I$, then $(TS)^{2}=(0)$, yielding a contradiction. Therefore, $I-TS-H-I$ is a cycle,  contrary to ${\rm gr}(\Bbb{AG}(R))=4$. Thus $TS=(0)$. If $TT_1=(0)$, then $I-T-T_1-I$ is a cycle, yielding a contradiction. So, $TT_1\neq (0)$. Similarly $SS_1\neq (0)$. Also $V_1\cap V_2=\emptyset$. Therefore, $\Bbb{AG}(R)\cong K_{|V_1|, |V_2|}$ and so $\Bbb{AG}(R)$ is a complete bipartite graph, and by Proposition \ref{cnew} $Z(R)$ is the union of two primes with intersection $\{(0)\}$.

${\rm (2)}\Rightarrow {\rm (1)}$ Clear.

${\rm (2)}\Leftrightarrow {\rm (3)}$ It follows from Proposition \ref{cnew}. \hfill$\square$
\end{pproof}

\begin{ttheo}\label{infty} 
The following statements are equivalent for a reduced ring $R$.

{\rm (1)} $\Bbb{AG}(R)$ is nonempty with ${\rm gr}(\Bbb{AG}(R))=\infty$.

{\rm (2)} There is a vertex which is adjacent to every vertex of $\Bbb{AG}(R)$.

{\rm (3)} $R \cong {\Bbb K} \times D$, where ${\Bbb K}$ is a field and $D$ is an integral domain.

{\rm (4)} $\Bbb{AG}(R)$ is a star graph.
\end{ttheo}

\begin{pproof}
$(1)\Rightarrow (2)$ Suppose to the contrary that there is not a vertex which is adjacent to every vertex of $\Bbb{AG}(R)$. Therefore, there exist distinct vertices $I_1, I_2, I_3, I_4$ such that $I_1-I_2-I_3-I_4$, $I_1I_3\neq (0)$, and $I_2I_4\neq (0)$. If $I_1I_4=(0)$, then $I_1-I_2-I_3-I_4-I_1$ is a cycle, contrary to ${\rm gr}(\Bbb{AG}(R))=\infty$. So we may assume that $I_1I_4\neq (0)$. Therefore, $I_1-I_2-I_3-I_1I_4-I_1$ is a cycle, contrary to ${\rm gr}(\Bbb{AG}(R))=\infty$. Therefore, there exists a vertex which is adjacent to every vertex of $\Bbb{AG}(R)$.

$(2)\Rightarrow (3)$ It follows from \cite[Corollary 2.3]{be-ra1}.

$(3)\Rightarrow (4)$ Clear.

$(4)\Rightarrow (1)$ Clear.\hfill$\square$
\end{pproof}

\begin{ttheo}
The following statements are equivalent for a non-reduced ring $R$.

{\rm (1)} $\Bbb{AG}(R)$ is nonempty with ${\rm gr}(\Bbb{AG}(R))=\infty$.

{\rm (2)} One of the following occurs:

{\rm (a)} Either $R\cong \frac{\Bbb{K}[x]}{(x^2)}$, where $\Bbb{K}$ is a field or $R\cong L$, where $L$ is a coefficient ring of 
characteristic $p^2$.

{\rm (b)} $R\cong R_1\times R_2$ such that $R_1$ is a field and either $R_2\cong \frac{\Bbb{K}[x]}{(x^2)}$, where $\Bbb{K}$ is a 
field or $R_2\cong L$, where $L$ is a coefficient ring of characteristic $p^2$.

{\rm (c)} $Z(R)$ is an annihilating ideal and if $IJ=(0)$ and $I\neq J$, then $I={\rm Ann}(Z(R))$ or $J={\rm Ann}(Z(R))$.

{\rm (3)} One of the following occurs:

{\rm (a)} $\Bbb{AG}(R)\cong K_1$.

{\rm (b)} $\Bbb{AG}(R)\cong P_4$. 

{\rm (c)} $\Bbb{AG}(R)\cong K_{1,n}$ for some $n\geq 1$.
\end{ttheo}

\begin{pproof}
(1) $\Rightarrow$ (2) Since $R$ is a non-reduced ring, there exists an ideal $I$ such that $I^2=(0)$. If $|\Bbb{I}(I)|\geq 4$, then there exist distinct ideals $I_1,I_2,I_3\in \Bbb{I}(I)$, such that $I_1-I_2-I_3-I_1$ is a cycle and so ${\rm gr}(\Bbb{AG}(R))=3$, yielding a contradiction. Thus without loss of generality we may assume that $I$ is a minimal ideal. We have the following cases:

$Case $ 1: There exists a minimal ideal $J$ such that $I\neq J$. Then either $J^2 = J$ or $J^2 = 0$. If $J^2=(0)$, then $I-J-(I+J)-I$ is a cycle, yielding a contradiction. So we may assume that $J^2 = J$. Thus by Brauer's
Lemma (see\cite[10.22]{lam}), $J = Re$ for some idempotent element $e\in R$, so $R= Re\oplus R(1-e)$. Therefore, $R\cong R_1\times R_2$.
Suppose that $|\Bbb{I}(R_1)|\geq 3$ and $|\Bbb{I}(R_2)|\geq 3$. Let $I_1$ be a nonzero proper ideal of $R_1$ and $I_2$ be a nonzero proper ideal of $R_2$. Then $(I_1,0)-(0,R_2)-(R_1,0)-(0,I_2)-(I_1,0)$ is a cycle, yielding a contradiction. So we may assume that either $|\Bbb{I}(R_1)|= 2$ or $\Bbb{I}(R_2)= 2$. Without loss of generality we assume that $|\Bbb{I}(R_1)|=2$ and so $R_1$ is a field. Since $R_1$ is a field and $R$ is a non-reduced ring, we conclude that $R_2$ is a non-reduced ring. Let $I_2$ and $J_2$ be nonzero ideals of $R_2$ such that $I_2J_2=(0)$. If $I_2\neq J_2$, then $(0,I_2)-(R_1,0)-(0,J_2)-(0,I_2)$ is a cycle, contrary to ${\rm gr}(\Bbb{AG}(R))=\infty$. Thus $|\Bbb{A}(R_2)|=3$. By \cite[Remark 10]{aal} either $R_2\cong \frac{\Bbb{K}[x]}{(x^2)}$, where $\Bbb{K}$ is a field or $R_2\cong L$, where $L$ is a coefficient ring of characteristic $p^2$.

$Case $ 2: $I$ is the unique minimal ideal of $R$. Suppose that there exists $K\in \Bbb{A}(R)^{*}$ such that $IK\neq (0)$. Since $K\in \Bbb{A}(R)^{*}$, there exists $J\in \Bbb{A}(R)^{*}$ such that $KJ=(0)$. If $IJ\neq (0)$, then since $I$ is minimal ideal, $IJ=I$. Hence $IK=(IJ)K=I(JK)=(0)$, yielding a contradiction. Therefore, $IJ=(0)$. Since $JK=(0)$ and $IK\neq (0)$, $I\nsubseteq J$. Since $I$ is the unique minimal ideal of $R$ and $I\nsubseteq J$, there exists $J_1\subseteq J$ such that $J_1\neq I$. Hence $I-J_1-K-J_2-I$ is a cycle and so ${\rm gr}(\Bbb{AG}(R))=3$, contrary to ${\rm gr}(\Bbb{AG}(R))=\infty$. Therefore, we must have $IK=(0)$ for every $K\in \Bbb{A}(R)^{*}$. Thus $IZ(R)=(0)$. Now, we have two subcases:

$Subcase $ 2-1: ${\rm Ann}(Z(R))\neq I$. If $|\Bbb{I}(Z(R))|\geq 4$, then there exists $S\in \Bbb{A}(R)^{*}$ such that $I\neq S\neq {\rm Ann}(Z(R))$ and so $I-{\rm Ann}(Z(R))-S-I$ is a cycle, yielding a contradiction. So we may assume that $|\Bbb{A}(R)|=3$. Thus $\Bbb{A}(R)^{*}=\{I, {\rm Ann}(Z(R))\}$.

$Subcase $ 2-2: $I={\rm Ann}(Z(R))$. If $|\Bbb{A}(R)^{*}|=1$, then by \cite[Remark 10]{aal} either $R\cong \frac{\Bbb{K}[x]}{(x^2)}$, where $\Bbb{K}$ is a field or $R\cong L$, where $L$ is a coefficient ring of characteristic $p^2$. So we may assume that $|\Bbb{A}(R)^{*}|\geq 2$. Let $S,J\in \Bbb{A}(R)^{*}$ such that $SJ=(0)$ and $S\neq J$. If $S\neq I$ and $J\neq I$, then $I-S-J-I$ is a cycle, yielding a contradiction. Therefore, $S=I={\rm Ann}(Z(R))$ or $J=I={\rm Ann}(Z(R))$.

(2) $\Rightarrow$ (3)  If either $R\cong \frac{\Bbb{K}[x]}{(x^2)}$, where $\Bbb{K}$ is a field or $R\cong L$, where $L$ is a coefficient ring of 
characteristic $p^2$, then $\Bbb{AG}(R)\cong K_1$. 

If $R\cong R_1\times R_2$ such that $R_1$ is a field and either $R_2\cong \frac{\Bbb{K}[x]}{(x^2)}$, where $\Bbb{K}$ is a  field or $R_2\cong L$, where $L$ is a coefficient ring of characteristic $p^2$, then $R_2$ has a non-trivial ideal say $I$, and $\Bbb{AG}(R)\cong (R_1,I)-(0,I)-(R_1,0)-(0,R_2)\cong P_4$.

 Let $Z(R)$ is an annihilating ideal and if $IJ=(0)$ ($I\neq J$), then $I={\rm Ann}(Z(R))$ or $J={\rm Ann}(Z(R))$. Then every annihilating ideal is only adjacent to $I$ and so either $\Bbb{AG}(R)\cong K_{1}$ or $\Bbb{AG}(R)\cong K_{1,n}$ for some $n\geq 1$.

(3) $\Rightarrow$ (1) Clear.\hfill$\square$
\end{pproof}

\begin{ttheo}
The following statements are equivalent for a non-reduced ring $R$.

{\rm (1)} $\Bbb{AG}(R)$ is nonempty with ${\rm gr}(\Bbb{AG}(R))=4$.

{\rm (2)} $R\cong R_1\times R_2$, where either $R_1\cong \frac{\Bbb{K}[x]}{(x^2)}$, where $\Bbb{K}$ is a field or $R_1\cong L$, where $L$ is a coefficient ring of characteristic $p^2$ and $R_2$ is an integral domain which is not a field.

{\rm (3)} $\Bbb{AG}(R)$ is isomorphic to Figure 1.
\end{ttheo}

\begin{pproof}
(1) $\Rightarrow$ (2) Since $R$ is a non-reduced ring, there exists an ideal $I$ such that $I^2=(0)$. If $|\Bbb{I}(I)|\geq 4$, then there exist distinct ideals $I_1,I_2,I_3\in \Bbb{I}(I)^{*}$, such that $I_1-I_2-I_3-I_1=(0)$ is a cycle and so ${\rm gr}(\Bbb{AG}(R))=3$, yielding a contradiction. Without loss of generality we may assume that $I$ is a minimal ideal.
We first show that there exist distinct ideals $I_1,I_2,I_3\in \Bbb{A}(R)^{*}$ such that $I_1-I_2-I_3-I-I_1$ is a cycle in $\Bbb{AG}(R)$. Since ${\rm gr}(\Bbb{AG}(R))=4$, there exist distinct ideals $I_1,I_2,I_3,I_4\in \Bbb{A}(R)^{*}$ such that $I_1-I_2-I_3-I_4-I_1$. Assume that $I=I_i$ for some $i$. Without loss of generality assume that $i=4$. Then $I_1-I_2-I_3-I-I_1$ is a cycle in $\Bbb{AG}(R)$. So we may assume that $I_i\neq I$ for all $1\leq i\leq 4$. If $I \nsubseteq I_i$ for all $1\leq i\leq4$, then $II_i=(0)$. Hence $I-I_1-I_2-I$ is a cycle, yielding a contradiction. Therefore, there exists $i$ such that $I\subseteq I_i$. Without loss of generality assume that $i=4$. Thus $I_1-I_2-I_3-I-I_1$ is a cycle in $\Bbb{AG}(R)$.
If $II_2=(0)$, then $I-I_1-I_2-I$ is a cycle in $\Bbb{AG}(R)$, yielding a contradiction. Thus $II_2\neq (0)$ and since $I$ is a minimal ideal, $I\subseteq I_2$. Suppose that ${\rm Ann}(I)\cap I_2\neq I$. If $({\rm Ann}(I)\cap I_2)\neq I_3$, then $I-({\rm Ann}(I)\cap I_2)-I_3-I$ is a cycle in $\Bbb{AG}(R)$, yielding a contradiction. If ${\rm Ann}(I)\cap I_2= I_3$, then $I-({\rm Ann}(I)\cap I_2)-I_1-I$ is a cycle in $\Bbb{AG}(R)$, yielding a contradiction. Thus we can assume that ${\rm Ann}(I)\cap I_2= I$. Let $0\neq z\in I$. Then $Rz=I$. Since $Rz\cong R/{\rm Ann}(z)$ and $Rz$ is a minimal ideal of $R$, we conclude that ${\rm Ann}(z)={\rm Ann}(I)$ is a maximal ideal. Since $II_2\neq (0)$, $I_2\nsubseteq {\rm Ann}(I)$. Therefore, ${\rm Ann}(I)+I_2=R$. Thus there exist $x\in {\rm Ann}(z)$ and $y\in I_2$ such that $x+y=1$. Since ${\rm Ann}(I)\cap I_2= I$, $(Rx)\cap (Ry)\subseteq I\subseteq nil(R)$. If $x\in nil(R)$, then there exists a positive integer $n$ such that $x^n=0$. Therefore, $(x+y)^n\in (Ry)$, contrary to $x+y=1$. Thus $x\not \in nil(R)$. Similarly $y\not \in nil(R)$. Note that $xy\in (Rx)\cap (Ry)\subseteq nil(R)$, we obtain that $x^{2}+nil(R)=(x^{2}+xy)+nil(R)= x(x+y)+nil(R)=x+nil(R)$. Thus $x+nil(R)$ is a nontrivial idempotent in $R/nil(R)$ and hence by \cite[Corollary, p.73]{lamb} $R$ has a nontrivial idempotent.
Since $R$ has a nontrivial idempotent, $R\cong R_1\times R_2$.
Note that $R$ is a non-reduced ring, so either $R_1$ or $R_2$ is a non-reduced ring. Without loss of generality assume that $R_1$ is a non-reduced ring. 
Suppose that $I_1$ and $I_2$ are ideals of $R_1$ such that $I_1I_2=(0)$. If $I_1\neq I_2$, then $(0,R_2)-(I
_1,0)-(I_2,0)-
(0,R_2)$
is a cycle
in $\Bbb{AG}(R)$, yielding a contradiction. Thus $I_1=I_2$. We conclude that $|\Bbb{A}(R_1)^*|=1$. Thus by \cite[Remark 10]{aal}, either $R_1\cong \frac{\Bbb{K}[x]}{(x^2)}$, where $\Bbb{K}$ is a field or $R_1\cong L$, where $L$ is a coefficient ring of characteristic $p^2$. We have the following cases:

$Case$ 1: $R_2$ is an integral domain. If $R_2$ is a field then it is easy to see that $\Bbb{AG}(R)$ is a star graph, yielding a contradiction since ${\rm gr}(\Bbb{AG}(R))=4$. Therefore, $R_2$ is an integral domain which is not a field.

$Case$ 2: $R_2$ is not an integral domain. Then there exist $I_2, J_2\in \Bbb{A}(R_2)^{*}$ such that $I_2J_2=(0)$. Since $|\Bbb{A}(R_1)^*|=1$, there exists $I_1\in \Bbb{A}(R_1)^{*}$ such that ${(I_1)}^{2}=(0)$. Thus $(I_1,0)-(I_1,J_2)-(0,I_2)-(I_1,0)$ is a cycle in $\Bbb{AG}(R)$, yielding a contradiction. Therefore, this case is impossible.

(2) $\Rightarrow$ (3)  Let $I$ be the only nontrivial ideal of $R_1$. Then $\Bbb{AG}(R)$ is isomorphic to Figure 1.

 $(3) \Rightarrow (1)$ Clear. \hfill$\square$
\end{pproof}

\begin{center}
\begin{tikzpicture}[scale=0.7]
\tikzstyle{every node}=[draw, shape=circle, inner sep=2pt];
\node (v1) at (1,0)[draw, circle, fill][label=above right:${(I,0)}$]{};
\node (v2) at (-1,0)[draw, circle, fill][label=above left:${(R_1,0)}$]{};
\node (v3) at (2,-6)[draw, circle, fill][label=right:${}$]{};
\node (v4) at (1,-6)[draw, circle, fill][label=right:${}$]{};
\node (v5) at (-1,-6)[draw, circle, fill][label=right:${}$]{};
\node (v6) at (-2,-6)[draw, circle, fill][label=right:${}$]{};
\node (v7) at (-3,-6)[draw, circle, fill][label=above right:${}$]{};
\node (v8) at (-4,-6)[draw, circle, fill][label=left:${\ldots~~~~}$]{};
\node (v9) at (4,-6)[draw, circle, fill][label=right:${~~~~\ldots}$]{};
\node (v10) at (3,-6)[draw, circle, fill][label=above right:${}$]{};
\node (v11) at (0,3)[draw, circle, fill][label=above:${(I,R_2)}$]{};

\draw (v11) --(v1) -- (v3) --(v2)--(v4) -- (v1)-- (v5)-- (v2)--(v6)--(v1)--(v7)--(v2)(v1)--(v8) --(v2)--(v9)--(v1)--(v10)-- (v2);
\end{tikzpicture}
\end{center}

\begin{center}
Figure 1
\end{center}

\section {A Relation Between the Smarandache Vertices, Girth, and Diameter of the Annihilating-ideal Graphs}

The concept of a {\it Smarandache vertex} in a (simple) graph was first
introduced by Rahimi \cite{rahim} in order to study the
{\it Smarandache zero-divisors} of a commutative ring which was
introduced by Vasantha Kandasamy in \cite{vas} for semigroups and
rings (not necessarily commutative). A non-zero element $a$ in a
commutative ring $R$ is said to be a Smarandache zero-divisor if
there exist three different nonzero elements $x$, $y$, and $b$
($\neq a$) in $R$ such that $ax = ab = by = 0$, but $xy \neq 0$.
This definition of a Smarandache zero-divisor (which was given in
\cite{rahim}) is slightly different from the definition of
Vasantha Kandasamy in \cite{vas}, where in her definition $b$
could also be equal to $a$.
In this section, we provide some examples and facts about the Smarandache vertices (or {\it S-vertices} for short) of $\Bbb {AG}(R)$. First, we define the notion of a Smarandache vertex in a simple graph and provide several (in particular, graph-theoretic) examples (see Lemmas \ref{triv1},
\ref{triv2}, and Proposition \ref{color}). Also we provide some
more ring-theoretic examples as well.

\begin{defi} A vertex $a$ in a simple graph $G$ is said to be a Smarandache vertex (or S-vertex for short) provided that there exist three distinct vertices $x$, $y$, and $b$ ($\neq a$) in $G$ such that $a$ ---$x$, $a$---$b$, and $b$---$y$ are edges in $G$; but there is no edge between $x$ and $y$. \end{defi}

Note that a graph containing a Smarandache vertex should have at
least four vertices and three edges, and also the degree of each
S-vertex must be at least 2. The proofs of the next two lemmas (Lemma \ref{triv1} and Lemma \ref{triv2}) are
not difficult and can be followed directly from the definition and
we leave them to the reader.
Recall that for a graph $G$, a complete subgraph of G is called a
{\it clique}. The {\it clique number}, $\omega(G)$, is the
greatest integer $n \geq 1$ such that $K_n \subseteq G$, and
$\omega(G)$ is infinite if $K_n \subseteq G$ for all $n \geq 1$.
The {\it chromatic number} $\chi(G)$ of a graph $G$ is defined to
be the minimum number of colors required to color the vertices of
$G$ in such a way that no two adjacent vertices have the same
color. A graph is called {\it weakly perfect} if its chromatic
number equals its clique number.

\begin{llem} \label{triv1} The following statements are true for the given graphs:
\begin{enumerate}
\item[$\mathrm{(1)}$] A complete graph does not have any S-vertices.

\item[$\mathrm{(2)}$] A star graph does not have any S-vertices.

\item[$\mathrm{(3)}$] A complete bipartite graph has no S-vertices.

\item[$\mathrm{(4)}$] Let $G$ be a complete $r$-partite graph ($r \geq 3$) with parts $V_1, V_2, \ldots, V_r$. If at least one part, say $V_1$, has at least two elements, then every element not in $V_1$ is an S-vertex. Further, if there exist at least two parts of $G$ such that each of which has at least two elements, then every element of $G$ is an S-vertex.

\item[$\mathrm{(5)}$] A bistar graph has two Smarandache vertices; namely, the center of each star. A {\it bistar} graph is a graph generated by two star graphs when their centers are joined.

\item[$\mathrm{(6)}$] Every vertex in a cycle of size greater than or equal to five in a graph is an S-vertex provided that there is no edge between the nonneighbouring vertices. In particular, every vertex in a cyclic graph $C_n$ of size larger than or equal to 5 is a Smarandache
vertex. Note that for odd integers $n \geq 5$, $\chi(C_n) = 3$ and $\omega(C_n) = 2$; and for even integers $n \geq 5$, $\chi(C_n) = \omega(C_n) = 2$.

\item[$\mathrm{(7)}$] Let $G$ be a graph containing two distinct vertices $x$ and $y$ such that $d(x,y) = 3$. Then $G$ has an S-vertex. But the converse is not true in general. Suppose $G$ is the graph $x$---$a$, $a$---$b$, $b$---$y$, and $a$---$y$; where obviously, $a$ is an S-vertex and $d(x,y) = 2$. Note that if diameter of $G$ is 3, then it has an S-vertex since there exist two distinct vertices $x$ and $y$ in $G$ such that $d(x,y) = 3$.
\end{enumerate}
\end{llem}

\begin{eexam} In \cite[Corollary 2.3]{be-ra1}, it is shown that for any reduced
ring $R$, $\Bbb {AG}(R)$ is a star graph if and only if $R \cong F
\times D$, where $F$ is a field and $D$ is an integral domain. In
this case, $\Bbb {AG}(R)$ has no Smarandache vertices.
\end{eexam}

\begin{eexam} In \cite[Lemma 1.8]{be-ra2} it is shown that for any reduced ring $R$ with finitely many minimal primes, ${\rm diam}(\Bbb{AG}(R)) = 3$ provided $R$ has more than two minimal
primes. Thus by Lemma \ref{triv1}(7), $\Bbb {AG}(R)$ has an
S-vertex. This also could be an example of a weakly perfect graph containing an S-vertex since by \cite[Corollary 2.11]{be-ra2}, $\Bbb {AG}(R)$ is weakly perfect for any reduced ring $R$ (see also Proposition \ref{color}, Remark \ref{weak1}, and Example \ref{weak2}). \end{eexam}

\begin{llem} \label{triv2} Let $C$ be a clique in a graph $G$ such that $|C| \geq 3$. Suppose that $x$ is a vertex in $G\setminus C$ and $x$ makes a link with at least one vertex or at most $|C|-2$ vertices of $C$, then every vertex of $C$ is an S-vertex. In other case, if $x$ makes links with $|C|-1$ vertices of $C$, then
all those $|C|-1$ vertices are S-vertices.
\end{llem}

\begin{ppro} \label{color} Let $G$ be a connected graph whose clique number is strictly larger than 2. If $\omega(G) \neq \chi(G)$, then $G$ has an S-vertex. In other words, for any connected graph $G$ with $\omega(G) \geq 3$ and no S-vertices, then $\omega(G) = \chi(G)$ (i.e., $G$ is weakly perfect). \end{ppro}

\begin{pproof} Let $C$ be a (largest) clique in $G$ with $|C| \geq 3$. Since $\omega(G) \neq \chi(G)$, then $G$ is not a complete graph. Thus, there exists a vertex $x \in G \setminus C$ which makes edge(s) with at least one or at most $\omega(G)-1$ elements of $C$. Now the proof is immediate from Lemma \ref{triv2}. \end{pproof}

\begin{rrem} \label{weak1} In the next example we show that The converse of the above proposition need not be true in general. Also None of the graphs in Parts (1), (2), and (3) of Lemma \ref{triv1}, has an S-vertex where $\omega(G) = \chi(G)$. Note that each of the graphs in Parts (2) and (3) has $\omega(G) = \chi(G) = 2$. The graph in Part (5) has two S-vertices and $\omega(G) = \chi(G) = 2$. See also Part (6) of Lemma \ref{triv1}. \end{rrem}

\begin{eexam} \label{weak2} As in \cite[Proposition 2.1]{be-ra2}, let
\begin{center}
$R=\Bbb{Z}_4[X, Y, Z]/(X^2-2, Y^2-2, Z^2, 2X, 2Y, 2Z, XY, XZ,
YZ-2)$\end{center} be a ring and $C = \{(2), (x), (y), (y+z)\}$ a
clique in $\Bbb {AG}(R)$. Since $(z) \notin C$ and it does not
make a link with all the elements of $C$, then by Lemma
\ref{triv2}, $C$ contains an S-vertex. Hence by \cite[Proposition
2.1]{be-ra2}, this is an example of a weakly perfect graph
containing a Smarandache vertex with $\chi(\Bbb{AG}(R)) =
\omega(\Bbb{AG}(R)) = 4 \geq 3$. \end{eexam}

\begin{rrem} Conjecture 0.1 in \cite{be-ra2} states that $\Bbb {AG}(R)$ is weakly perfect for any ring $R$. Now from Proposition \ref{color}, this conjecture is true for any ring $R$ with $\omega(\Bbb{AG}(R)) \geq 3$ and $\Bbb{AG}(R)$ containing no S-vertices. Note that \cite[Corollary 2.11]{be-ra2} proves the validity of this conjecture for any reduced ring $R$. \end{rrem}

\begin{ppro} Let $\{I_1, I_2, \ldots, I_n\}$ be a clique in $\Bbb {AG}(R)$ with $n \geq 3$. Then
\begin{enumerate}
\item[$\mathrm{(1)}$] $\Bbb {AG}(R)$ contains $n$ S-vertices provided that $I_i^2 \neq (0)$ and $I_j^2 \neq (0)$ for some $1 \leq i \neq j \leq n$.

\item[$\mathrm{(2)}$] $\Bbb {AG}(R)$ contains $n$ S-vertices provided that $I_i^2 \neq (0)$ and $I_j \not \subseteq I_i$ for some $1 \leq i \neq j \leq n$.

\item[$\mathrm{(3)}$] $\Bbb {AG}(R)$ contains $n$ S-vertices provided that $I_j^2 \neq (0)$ and $I_i \not \subseteq I_j$ for some $1 \leq i \neq j \leq n$.

\item[$\mathrm{(4)}$] $\Bbb {AG}(R)$ contains $n$ S-vertices provided that $R$ is a reduced ring.
\end{enumerate}
\end{ppro}

\begin{pproof} We just prove Part (1) and leave the other parts to the reader. Without loss of generality suppose that $I_1^2 \neq (0)$ and $I_2^2 \neq (0)$. Now the proof follows from Lemma \ref{triv2} and the fact that $I_1+I_2$ is a vertex different from all vertices of the clique and makes a link with each of them except $I_1$ and $I_2$. Note that $I_1+I_2 \neq R$. Otherwise, $I_3 = I_3R = I_3I_1+I_3I_2 = (0)$ which is a contradiction. \end{pproof}

\begin{llem} \label{direct} Let $R = R_1 \times R_2 \times \cdots \times R_n$ be the direct product of $n \geq 2$ rings.
If $\Bbb {AG}(R)$ has no S-vertices, then $n = 2$ and $R = R_1 \times R_2$, where each of the rings $R_1$ and $R_2$ is an integral domain.
\end{llem}

\begin{pproof}Without loss of generality suppose $n = 3$. Let $C = \{I_1, I_2, I_3\}$, where $I_1 = R_1 \times (0) \times (0)$, $I_2 = (0) \times R_2 \times (0)$, and $I_3 = (0) \times (0) \times R_3$. Clearly $C$ is a clique in $\Bbb {AG}(R)$. Let $A = (0) \times R_2 \times R_3$. Now Lemma \ref{triv2} implies the existence of an S-vertex in $\Bbb {AG}(R)$ which is a contradiction. Hence $n = 2$ and $R = R_1 \times R_2$.\\
\noindent Now suppose that $R_2$ is not an integral domain. Thus, there exist two nonzero proper ideals $I$ and $J$ in $R_2$ such that $IJ = (0)$. Therefore, \begin{center}
$(0, R_2)$---$(R_1, 0)$---$(0, I)$---$(R_1,J)$ \end{center}
implies the existence of an S-vertex, yielding a contradiction. Thus $R_2$ and similarly $R_1$ are integral domains.
 \end{pproof}

\begin{ppro} Let $R$ be a commutative ring. Then \begin{enumerate}
\item[$\mathrm{(1)}$] If $R$ is a non-local Artinian ring, then $\Bbb {AG}(R)$ has no S-vertices if and only if $R = F_1 \times F_2$ where each of $F_1$ and $F_2$ is a field.

\item[$\mathrm{(2)}$] Let $R$ be an Artinian ring with ${\rm gr}(\Bbb {AG}(R)) = 4$. Then $R$ can not be a local ring.
 \end{enumerate}
\end{ppro}

\begin{pproof} Part (1) is an immediate consequence of Proposition \ref{direct} and the fact that any Artinian ring is a finite direct product of local rings \cite[Theorem 8.7]{ati}. For Part (2), suppose $(R, M)$ is an Artinian local ring. Thus by \cite[Theorem 82]{kap}, $M = {\rm Ann}(x)$ for some $0 \neq x \in M$. Hence $I = Rx$ is an ideal which is adjacent to every nonzero proper ideal of $R$. Now since $\Bbb {AG}(R)$ contains a cycle, there exist two vertices $J$ and $K$ such that $I-J-K-I$. This is impossible since ${\rm gr}(\Bbb {AG}(R)) = 4$. Thus $R$ can not be a local ring.
\end{pproof}

\begin{llem} \label{red}
Let $R$ be a reduced ring such that $\Gamma(R)$ contains an S-vertex. Then $\Bbb {AG}(R)$ has an S-vertex. Thus the number of S-vertices of $\Gamma(R)$ is less than or equal to the number of S-vertices of $\Bbb {AG}(R)$ for any reduced ring.
\end{llem}

\begin{pproof} Let $a-x-y-b$ be a path of length 3 in $\Gamma(R)$ such that $x$ is an S-vertex in $\Gamma(R)$. Clearly $ab \neq 0$ by definition. Thus $Ra-Rx-Ry-Rb$ is a path of length 3 in $\Bbb {AG}(R)$ since $R$ is reduced. Therefore $Rx$ is an S-vertex in $\Bbb {AG}(R)$ by definition. \end{pproof}

\begin{ttheo} The following are true for a reduced ring $R$.
\begin{enumerate}
\item[$\mathrm{(1)}$] Assume $R$ contains $k \geq 3$ distinct minimal prime ideals. Then each of $\Gamma(R$) and $\Bbb {AG}(R)$ has an S-vertex.

\item[$\mathrm{(2)}$] Let $Z(R)$ be the union of two primes with intersection $(0)$. Then $\Bbb {AG}(R)$ has no S-vertices.

\item[$\mathrm{(3)}$] If ${\rm gr}(\Bbb {AG}(R)) = 4$, then $\Bbb {AG}(R)$ has no S-vertices.

\item[$\mathrm{(4)}$] Suppose that $\Bbb{AG}(R)$ is nonempty with ${\rm gr}(\Bbb{AG}(R))=\infty$. Then $\Bbb{AG}(R)$ has no S-vertices.

\end{enumerate}
\end{ttheo}

\begin{proof} We just prove Part (1) since the other three parts are immediate from Lemma \ref{triv1} and Proposition \ref{cnew}, Theorem \ref{3.4}, and Theorem \ref{infty} respectively. Since $R$ is reduced, then $nil(R) = (0) = \cap P_i$ for $1 \leq i \leq k$, where $nil(R)$ is the ideal of all nilpotent elements of $R$. Let $a_i$ be in $P_i \setminus \cup P_j$ for all $1 \leq j \neq i \leq k$. Clearly $a_1a_2a_3 \cdots a_k = 0$. Let $x = a_2a_3 \cdots a_k$ and $y = a_1a_3a_4 \cdots a_k$. Now by hypothesis, it is easy to see that $a_1$, $x$, $y$, and $a_2$ are all distinct and nonzero elements of $R$ and $a_1x = xy = ya_2 = 0$ with $a_1a_2 \neq 0$. Therefore, $x$ and $y$ are S-vertices in $\Gamma(R)$. Now the proof is complete by Lemma \ref{red}. \end{proof}

\begin{rrem} From Lemma \ref{triv1}(7), it is clear that if $\Gamma(R)$ [resp. $\Bbb {AG}(R)$] contains no S-vertices, then ${\rm diam}(\Gamma(R)) \neq 3$ [resp. ${\rm diam}(\Bbb {AG}(R)) \neq 3$]. In other words, ${\rm diam}(\Gamma(R)) \leq 2$ [resp. ${\rm diam}(\Bbb {AG}(R)) \leq 2$] since the diameter of each of these graphs is less than or equal to 3. Also, Proposition 1.1 of \cite{be-ra2} provides a relation between the diameters of $\Gamma(R)$ and $\Bbb {AG}(R)$. Consequently, combining the results of \cite[Proposition 1.1]{be-ra2} and existence (nonexistence) of S-vertices of these graphs may provide a relation between the S-vertices and diameters of $\Gamma(R)$ and $\Bbb {AG}(R)$. For example, if $\Bbb {AG}(R)$ contains no S-vertices, then ${\rm diam}(\Bbb {AG}(R)) \neq 3$ which by \cite[Proposition 1.1(d)]{be-ra2}, it implies ${\rm diam}(\Gamma(R)) \neq 3$. Notice that \cite[Proposition 1.1(d)]{be-ra2} states that if ${\rm diam}(\Gamma(R)) = 3$, then ${\rm diam}(\Bbb {AG}(R)) = 3$. \end{rrem}

\begin{ttheo} \label{bipart}
The following are true for a commutative ring $R$.
\begin{enumerate}
\item[$\mathrm{(1)}$] Let ${\rm gr}(\Bbb {AG}(R)) = 4$ and $I-J-K-L-I$ be a cycle in $\Bbb {AG}(R)$ such that $I^2 \neq 0$. Then $\Bbb {AG}(R)$ is complete bipartite when $\Bbb {AG}(R)$ has no S-vertices.

\item[$\mathrm{(2)}$] If $\Bbb {AG}(R)$ is complete bipartite, then $(\Bbb {AG}(R)$ has no S-vertices with ${\rm gr}(\Bbb {AG}(R)) = 4$ or $\infty$.
\end{enumerate}
\end{ttheo}

\begin{pproof}
We just give a proof for Part (1) since the other part is obvious. Clearly, ${\rm diam}(\Bbb{AG}(R)) \neq 3$ since $\Bbb {AG}(R)$ has no S-vertices. If ${\rm diam}(\Bbb{AG}(R))=0$ or $1$, then $\Bbb{AG}(R)$ is a complete graph and so ${\rm gr}(\Bbb{AG}(R))$ is $3$ or $\infty$, yielding a contradiction. Therefore, ${\rm diam}(\Bbb{AG}(R))=2$. We now show that $\Bbb{AG}(R)$ is a complete bipartite graph. Since ${\rm gr}(\Bbb{AG}(R))=4$, there exist $I,J,K,L\in \Bbb{AG}(R)$ such that $I-J-K-L-I$ with $I^2 \neq (0)$ by hypothesis. We show that $\Bbb{AG}(R)\cong K_{|V_1|, |V_2|}$, where $V_1=\{T\in \Bbb{A}(R)^{*}: T\subseteq {\rm Ann}(I)\}$ and $V_2=\{S\in \Bbb{A}(R)^{*}: S\nsubseteq {\rm Ann}(I)\}$. Let $T, T_1\in V_1$ and $S, S_1\in V_2$. Then $IT=(0)$ and $IS\neq (0)$. Assume that $TS\neq (0)$. Since ${\rm diam}(\Bbb{AG}(R))=2$, there exists $H\in \Bbb{A}(R)^*$ such that $I-H-S$. Clearly, $TS \neq (0)$ implies that $T$ is not contained in $H$ and $T \neq H$. If $TH = (0)$, then ${\rm gr}(\Bbb {AG}(R)) = 3$, yielding a contradiction. Also $T$ is not a proper subset of $S$ since $TH \neq (0)$. Thus $I$ is an S-vertex in $\Bbb {AG}(R)$ which is a contradiction. Therefore $TS=(0)$. If $TT_1=(0)$, then $I-T-T_1-I$ is a cycle, yielding a contradiction. So, $TT_1\neq (0)$. Similarly $SS_1\neq (0)$. Also $V_1\cap V_2=\emptyset$. Therefore, $\Bbb{AG}(R)\cong K_{|V_1|, |V_2|}$ and so $\Bbb{AG}(R)$ is a complete bipartite graph. \end{pproof}


\begin{thebibliography}{1}
\bibitem{aal1} G. Aalipour, S. Akbari, M. Behboodi, R. Nikandish, M. J. Nikmehr and F. Shahsavari,
On the coloring of the annihilating-ideal graph of a
commutative ring, {\it Discerete Math.} {\bf 312} (2012) 2620-2626 .\vspace{-3mm}
\bibitem{aal} G. Aalipour, S. Akbari, M. Behboodi, R. Nikandish, M. J. Nikmehr and F. Shahsavari,
The classification of the annihilating-ideal graph of a
commutative ring,  {\it Algebra Colloq.} {\bf 21} (2014)  249–256.\vspace{-3mm}
\bibitem{abmr} F. Aliniaeifard, M. Behboodi, E. Mehdi-Nezhad, A. M. Rahimi, The Annihilating-Ideal Graph of a Commutative Ring with Respect to an Ideal, {\it Comm.  Algebra} {\bf 42} (2014)  2269–2284.\vspace{-3mm}
\bibitem{an-li2} D. F. Anderson, A. Frazier, A. Lauve, P. S.
Livingston, The Zero-Divisor Graph of a Commutative Ring II,
Lecture Notes in Pure and Appl. Math., 220, Dekker, New York,
2001.\vspace{-3mm}
\bibitem{an-li} D. F. Anderson, P. S. Livingston, The zero-divisor graph of a
commutative ring, {\it J. Algebra} {\bf 217} (1999)
434-447.\vspace{-3mm}
\bibitem{axt} M. Axtell, J. Coykendall, J. Stickles, Zero-divisor graphs
of polynomials and power series over commutative rings, {\it Comm.
Algebra} {\bf 33} (2005) 2043-2050.\vspace{-3mm}
\bibitem{beck} I. Beck, Coloring of commutative rings, {\it J. Algebra} {\bf 116}
(1988) 208-226.\vspace{-3mm}
\bibitem{be-ra1} M. Behboodi, Z. Rakeei, The annihilating-ideal graph of
commutative rings I, {\it J. Algebra
Appl.} {\bf 10} (2011) 727-739.\vspace{-3mm}
\bibitem{be-ra2} M. Behboodi, Z. Rakeei, The annihilating-ideal graph of
commutative rings II,  {\it J. Algebra
Appl.} {\bf 10} (2011) 740-753.\vspace{-3mm}
\bibitem{lam} T. Y. Lam, {\it A first course in noncommutative rings},
Springer-Verlag New York, Inc 1991.\vspace{-3mm}
\bibitem{lamb} J. Lambeck, {\it Lectures on Rings and Modules}, Blaisdell Publishing Company, Waltham, Toronto, London,
1966.\vspace{-3mm}
\bibitem{kap} I. Kaplansky, {\it Commutative Rings}, rev. ed. Chicago: Univ. of
Chicago Press, 1974.\vspace{-3mm}
 \bibitem{rahim} A. M. Rahimi,  Smarandache Vertices of the Graphs Associated to the Commutative Rings, {\it Comm. Algebra} {\bf 41} (2013), 1989-2004.\vspace{-3mm}
\bibitem{vas} W. B. Vasantha Kandasamy, {\it Smarandache Zero Divisors}, (2001)
 http://www.gallup.unm.edu/smarandache/zero.\\divisor.pdf.

\end{thebibliography}
\end{document}